\documentclass{amsart}

\usepackage{xcolor}
\usepackage{setspace}

\newtheorem{theorem}{Theorem}[section]
\newtheorem{prop}[theorem]{Proposition}
\newtheorem{lem}[theorem]{Lemma}

\newtheorem{defn}[theorem]{Definition}

\DeclareMathOperator*{\esssup}{ess\,sup}

\DeclareMathOperator*{\essinf}{ess\,inf}

\DeclareMathOperator{\osc}{{\rm osc}}

\numberwithin{equation}{section}
\title[Variational P\'{o}lya-Seg\H{o} principle]{A variation on the P\'{o}lya-Seg\H{o} principle in one dimension}

\author{Martin Lind}
\address{Department of Mathematics and Computer Science, Karlstad University, Universitetsgatan 2, 65188 Karlstad, Sweden}
\email{martin.lind@kau.se}

\author{Sorina Barza}
\address{Department of Mathematics and Computer Science, Karlstad University, Universitetsgatan 2, 65188 Karlstad, Sweden}
\email{sorina.barza@kau.se}

\subjclass[2020]{26A45, 46E35}

\keywords{non-increasing rearrangement, P\'{o}lya–Szeg\H{o} principle, bounded $p$-variation, fractional smoothness, modulus of continuity}

\begin{document}

\begin{abstract}
    We establish a one-dimensional P\'{o}lya–Szeg\H{o} principle for the Riesz $(p,\alpha)$-variation $\mathcal{V}_p^\alpha$, a family of functionals that interpolates between Wiener $p$-variation and the Sobolev seminorm $\|f'\|_{L^p}$. More precisely, for every measurable function $f$, we prove that its non-increasing rearrangement $f^*$ satisfies 
    $$
    \mathcal{V}_p^\alpha(f^*)\le\mathcal{V}_p^\alpha(f)
    $$ 
    for all $1\le p<\infty$ and $0\le\alpha\le 1-1/p$.
    This result contains and extends the classical variation-diminishing property of rearrangements from Sobolev spaces to a scale of spaces admitting fractional smoothness and, in particular, applies to functions that are nowhere differentiable. Our methods also yield a sharp rearrangement inequality for a modulus of continuity defined via $p$-variation, providing an analogue of a problem posed by Ul'yanov. Finally, we sketch how our inequality can be useful in the study of nonlinear Fredholm integral equations.
\end{abstract}
\maketitle

\section{Introduction}

Let $I=[a,b]\subset\mathbb{R}$ and let $f:I\rightarrow\mathbb{R}$. For any subset $E\subset I$ define
\begin{equation}
    \nonumber
    {\rm osc}(f;E)=\sup_{x\in E}f(x)-\inf_{x\in E}f(x).
\end{equation}
Denote also $|E|$ the Lebesgue measure of the set $E$.
\begin{defn}
    Let $1\le p<\infty$ and 
    \begin{equation}
        \nonumber
        p'=\frac{p}{p-1}\quad(p>1),\quad 1'=\infty.
    \end{equation}
    For $0\le\alpha\le1/p'$  the \emph{Riesz $(p,\alpha)$-variation} of $f$ on the interval $I$ is defined by
    \begin{equation}
        \label{p-alpha-var}
        \mathcal{V}_p^\alpha(f;I)=\sup_\mathcal{I}\left(\sum_{I_j\in\mathcal{I}}\left(\frac{{\rm osc}(f;I_j)}{|I_j|^{\alpha}}\right)^p\right)^{1/p},
    \end{equation}  
    where the supremum is taken over all families $\mathcal{I}=\{I_j\}$ of non-overlapping intervals contained in $I$. We also set
    \begin{equation}
        \nonumber
        \mathcal{BV}_p^\alpha(I)=\{f:I\rightarrow\mathbb{R}:\mathcal{V}_p^\alpha(f;I)<\infty\}.
    \end{equation}
\end{defn}
The parameter $\alpha$ should be interpreted as an index of fractional smoothness. In the case of zero smoothness, the quantity defined in (\ref{p-alpha-var}) reduces to the classical total variation when $p=1$ and to the total $p$-variation in the sense of Wiener \cite{Wi24} when $p>1$ (see also \cite{Yo36}). In this case, we use the notation
\begin{equation}
    \nonumber
    \mathcal{V}_p^0(f;I)=\mathcal{V}_p(f;I),\quad\mathcal{BV}_p^0(I)=\mathcal{BV}_p(I).
\end{equation}
Seminorms of the type (\ref{p-alpha-var}) were studied in e.g. \cite{KoLi09,Li13}.
For $1<p<\infty$ and $\alpha=1/p'$, a theorem of F. Riesz states that 
\begin{equation}
    \nonumber
    \mathcal{V}_p^{1/p'}(f;I)=\sup_{\mathcal{I}}\left(\sum_{I_j\in\mathcal{I}}\frac{{\rm osc}(f;I_j)^p}{|I_j|^{p-1}}\right)^{1/p}=\|f'\|_{L^p(I)}.
\end{equation}
For a discussion of F. Riesz's theorem, see the monograph \cite[Section 3.5]{ApBaMe14} (see also \cite{BaLi15} and the references given therein). Thus, for $p\in(1,\infty)$ and $\alpha\in[0,1/p']$, the spaces $\mathcal{BV}_p^\alpha(I)$ form a scale of spaces of fractional smoothness connecting $\mathcal{BV}_p$ and the Sobolev space $W^{1,p}(I)$ (i.e. the functions $f\in L^p(I)$ with weak derivative $f'\in L^p(I))$. These spaces were investigated from this point of view in \cite{Li13}.
Note in passing that if $p=1$, then $1/p'=0$, forcing $\alpha=0$. This reflects the fact that functions in $\mathcal{BV}_1(I)$ do not possess any fractional smoothness in the sense defined above.

For a measurable function $f:I\rightarrow\mathbb{R}$, we denote by $f^*:I^*\rightarrow[0,\infty)$ the \emph{non-increasing rearrangement of $f$}. (See Section \ref{Aux} for the definition of $f^*$ and $I^*$.)

A remarkable property of $f^*$  is its variation-diminishing property expressed by the classical \emph{P\'{o}lya-Szeg\H{o} principle} \cite{PoSz51} (a similar result holds in higher dimensions): 
\begin{equation}
    \label{polya-szego-classical}
    \|(f^*)'\|_{L^p(I^*)}\le\|f'\|_{L^p(I)}\quad(1\le p<\infty).
\end{equation}
The main result of this note is that the functionals (\ref{p-alpha-var}) have a similar variation-diminishing property.
\begin{theorem}
    \label{mainTeo}
    For $1\le p<\infty$ and $0\le\alpha\le1/p'$ there holds
    \begin{equation}
    \label{polyaszegoIneq}
    \mathcal{V}_p^\alpha(f^*;I^*)\le\mathcal{V}_p^\alpha(f;I)
    \end{equation}
    for any measurable function $f:I\rightarrow\mathbb{R}$.
\end{theorem}
For discussions of other types of fractional P\'{o}lya-Szeg\H{o} principles, see e.g. \cite{Ko07} and \cite{Ca26}. 

In Section \ref{AddComm}, we discuss related and additional results. 
One interesting feature of (\ref{polyaszegoIneq}) is that, unlike (\ref{polya-szego-classical}), it applies to rough functions, including certain nowhere differentiable functions such as the Takagi-van der Waerden function and Riemann's ``nondifferentiable'' function. We discuss these questions. We also adapt the proof of Theorem \ref{mainTeo} to obtain a sharp estimate for the modulus of continuity of the non-increasing rearrangement $f^*$ with respect to the $\mathcal{BV}_p$-norm, expressed in terms of the corresponding modulus of continuity of $f$. This may be regarded as a $\mathcal{BV}_p$-analogue of a line of research initiated by Ul'yanov, see e.g. \cite{Ko07}.
Finally, we briefly indicate how (\ref{polyaszegoIneq}) can be applied to a class of nonlinear Fredholm integral equations.

\section{Auxiliary results}
\label{Aux}

\subsection{Non-increasing rearrangements}

Let $f:I\rightarrow\mathbb{R}$ be measurable. The \emph{distribution function} of $f$ is defined by
\begin{equation}
    \nonumber
    \mu_f(y)=|\{x\in I:|f(x)|>y\}|\quad(y\ge0).
\end{equation}
The  \emph{non-increasing rearrangement} $f^*$ of $f$ is defined on $I^*:=[0,b-a]$ by
\begin{equation}
    \nonumber
    f^*(t)=\begin{cases}
        \esssup_{x\in I}|f(x)|,&t=0\\
        \inf\{y>0:\mu_f(y)<t\},&0<t<b-a\\
        \essinf_{x\in I}|f(x)|,&t=b-a.
    \end{cases}
\end{equation}
The functions $|f|$ and $f^*$ are \emph{equimeasurable}: for any $s\ge0$ there holds
\begin{equation}
    \label{equimeasurable}
    |\{x\in I:|f(x)|>s\}|=|\{t\in I^*:f^*(t)>s\}|.
\end{equation}
For discussions of the theory and applications of non-increasing rearrangements, see \cite{BeSh88, Ko07}.

We shall need the following simple lemma.
\begin{lem}
    \label{decreasingLem}
    For any $y\ge0$
    \begin{equation}
        \label{eqlevelset}
        |\{x\in I:|f(x)|=y\}|=|\{t\in I^*:f^*(t)=y\}|.
    \end{equation}
    In particular, $f^*$ is strictly decreasing on $I^*$ if and only if 
    $$
    |\{x\in I:|f(x)|=y\}|=0
    $$
    for every $y\ge0$.
\end{lem}
\begin{proof}
    To prove (\ref{eqlevelset}) for $y>0$, take $n\in\mathbb{N}$ large enough to have $y-1/n\ge0$. Then
    \begin{equation}
        \nonumber
        |\{x\in I:y-1/n<|f(x)|<y\}|=\mu_f(y-1/n)-\mu_f(y)-|\{x\in I:|f(x)|=y\}|.
    \end{equation}
    Since 
    $$
    \bigcap_{n=1}^\infty\{x\in I:y-1/n<|f(x)|<y\}=\emptyset
    $$
    we have
    $$
    \lim_{n\rightarrow\infty}|\{x\in I:y-1/n<|f(x)|<y\}|=0,
    $$
    and consequently for every $y>0$ there holds
    \begin{equation}
        \label{limit}
        |\{x\in I:|f(x)|=y\}|=\lim_{n\rightarrow\infty}(\mu_f(y-1/n)-\mu_f(y)).
    \end{equation}
    By (\ref{equimeasurable}),
    \begin{eqnarray}
        \nonumber
        |\{x\in I:|f(x)|=y\}|&=&\lim_{n\rightarrow\infty}(\mu_{f^*}(y-1/n)-\mu_{f^*}(y))\\
        \nonumber
        &=&|\{t\in I^*:f^*(t)=y\}|,
    \end{eqnarray}
    where the last equation follows by applying (\ref{limit}) to $f^*$. Similarly, for $y=0$,
    \begin{equation}
        \nonumber
        |\{x\in I:f(x)=0\}|=|I|-\mu_f(0)=|I^*|-\mu_{f^*}(0)=|\{t\in I^*:f^*(t)=0\}|,
    \end{equation}
    proving (\ref{eqlevelset}) for every $y\ge0$.
    To prove the second statement, simply note that $f^*$ is strictly decreasing on $I^*$ if and only if $f^*$ has no interval where $f^*$ is constant which is equivalent to the condition that $\{t\in I^*:f^*(t)=y\}$ has measure 0 for every $y\ge0$.
\end{proof}

\subsection{Some results from real analysis}

Denote by $C(I)$ the space of continuous functions on $I$. Let $\alpha>0$ and denote by ${\rm Lip}_\alpha(I)$ the subspace of $C(I)$ consisting of all functions $f:I\rightarrow\mathbb{R}$ for which there exists a constant $C>0$ such that for every $x,y\in I$
\begin{equation}
    \nonumber
    |f(x)-f(y)|\le C|x-y|^\alpha.
\end{equation}
The space ${\rm Lip}_\alpha(I)$ is endowed with the following seminorm
\begin{equation}
    \nonumber
    |f|_{{\rm Lip}_\alpha(I)}=\sup_{x,y\in I, x\neq y}\frac{|f(x)-f(y)|}{|x-y|^\alpha}.
\end{equation}
For $\alpha=1,$ we write the space ${\rm Lip}_1(I)={\rm Lip}(I)$. For $\alpha>1$, it is easy to see that ${\rm Lip}_\alpha(I)$ consists only of constant functions. The next proposition relates the spaces ${\rm Lip}_\alpha(I)$ to $\mathcal{BV}_p^\alpha(I)$. The proof is immediate.
\begin{prop}
    \label{contProp}
    Assume that $1<p<\infty$ and $\alpha\in(0,1/p']$
    \begin{equation}
        \nonumber
        {\rm Lip}_{\alpha+1/p}(I)\hookrightarrow\mathcal{BV}_p^\alpha(I)\hookrightarrow{\rm Lip}_\alpha(I).
    \end{equation}
    Here, $Y\hookrightarrow X$ means that the space $Y$ is continuously embedded in the space $X$. 
\end{prop}
The next lemma is of a well-known type. It plays an important role in the proof of Theorem \ref{mainTeo}, hence we provide a proof.
\begin{lem}
    \label{mainLemma}
    Let $f:I\rightarrow\mathbb{R}$ be a continuous function. Take $y',y''\in f(I)$, $y'<y''$, and consider the set
    $$
    E=\{x\in I:y'<f(x)<y''\}.
    $$
    Then there exists an interval $J=(x',x'')\subset E$ such that
    \begin{equation}
    \label{extremal}
        \inf_{x\in J}f(x)=y'\quad\text{and}\quad\sup_{x\in J}f(x)=y''.
    \end{equation}
\end{lem}
\begin{proof}
    By the intermediate value theorem, there are $a_0,b_0\in I$ such that $f(a_0)=y'$ and $f(b_0)=y''$. Without loss of generality, we may assume $a_0<b_0$. Set
    $$
        A=\{x\in[a_0,b_0]:f(x)=y'\}.
    $$
    Since $a_0\in A$ and $A$ is bounded from above by $b_0$, the supremum $x'=\sup A$ is well-defined. Set
    $$
        B=\{x\in[x',b_0]:f(x)=y''\}.
    $$
    Since $b_0\in B$ and $B$ is bounded from below by $x'$, the infimum $x''=\inf B$ is well-defined. Clearly $x''\ge x'$. We shall now prove the following:
    \begin{enumerate}
        \item\label{lemP1} $x'<x''$;
        \item\label{lemP2} $f(x')=y'$, $f(x'')=y''$;
        \item\label{lemP3} $J:=(x',x'')$ is a subset of $E$;
        \item\label{lemP4} the relations (\ref{extremal}) hold.
    \end{enumerate}
    We start with (\ref{lemP1}). Assume for a contradiction that $x'=x''$. Take $\{x'_n\}\subset A$ with $x'_n\rightarrow x'$ and $\{x''_n\}\subset B$ with $x''_n\rightarrow x''$. If $x'=x''$, then by continuity of $f$, $\lim f(x'_n)=\lim f(x''_n)$. On the other hand, for every $n\in\mathbb{N}$, there holds
    $$
    f(x'_n)=y'\quad{\rm and}\quad f(x''_n)=y''
    $$
    whence $y'=y''$, contradicting $y'<y''$. Thus, $x'<x''$. Property (\ref{lemP2}) follows from the fact that $A$ and $B$ are closed and bounded sets, and therefore $x'\in A$ and $x''\in B$. 
    We proceed to show (\ref{lemP3}). Fix any $x\in(x',x'')$. First note that $f(x)\le y'$ is impossible. Indeed, assume that $f(x)\le y'$.
    Since $f(x'')=y''>y'$, the intermediate value theorem guarantees the existence of $c\in[x,x'')$ such that $f(c)=y'$, i.e. $c\in A$. But $c\ge x>x'$ and $c\in A$ contradicts the fact that $x'=\sup A$. Hence, $f(x)>y'$.
    Similarly, assume that $f(x)\ge y''$, then there must be $c\in(x',x]$ such that $f(c)=y''$, i.e., $c\in B$. But $c\le x< x''$ and $c\in B$ contradicts the fact that $x''=\inf B$. Consequently, for $x\in(x',x'')$ we have $y'<f(x)<y''$ whence $(x',x'')\subset E$.
    Finally, (\ref{lemP4}) is immediate from continuity of $f$, (\ref{lemP2}) and (\ref{lemP3}).
\end{proof}
Finally, we shall use the following lemma, expressing lower semicontinuity of the Riesz $(p,\alpha)$-variation.
\begin{lem}
    \label{semicontinuityLem}
    Assume that $\{f_n\}$ is a sequence of real-valued functions defined on $I$ that converges pointwise for every $x\in I$ to $f:I\rightarrow\mathbb{R}$. Then
    \begin{equation}
        \label{semicontinuity}
        \mathcal{V}^\alpha_p(f;I)\le\liminf_{n\rightarrow\infty}\mathcal{V}_p^\alpha(f_n;I).
    \end{equation}
\end{lem}
\begin{proof}
    Take an arbitrary finite family of disjoint intervals $\{I_k:1\le k\le N\}$ contained in $I$. 
    Fix $k$ and $x,y\in I_k$. Since $f_n\rightarrow f$ pointwise,
    \begin{equation}
        \nonumber
        f(x)-f(y)=\lim_{n\rightarrow\infty}(f_n(x)-f_n(y))\le\liminf_{n\rightarrow\infty}\osc(f_n;I_k).
    \end{equation}
    Taking supremum over $x\in I_k$ and infimum over $y\in I_k$ gives   
    \begin{equation}
        \nonumber
        \osc(f;I_k)\le\liminf_{n\rightarrow\infty}\osc(f_n;I_k).
    \end{equation}
    It follows that
    \begin{eqnarray}
        \nonumber
        \sum_{k=1}^N\frac{\osc(f;I_k)^p}{|I_k|^{\alpha p}}&\le&\sum_{k=1}^N\frac{\liminf_{n\rightarrow\infty}\osc(f_n;I_k)^p}{|I_k|^{\alpha p}}\\
        \nonumber
        &\le&\liminf_{n\rightarrow\infty}\sum_{k=1}^N\frac{\osc(f_n;I_k)^p}{|I_k|^{\alpha p}}\le\liminf_{n\rightarrow\infty}\mathcal{V}_p^\alpha(f_n;I)^p,
    \end{eqnarray}
    thus proving (\ref{semicontinuity}).
    \end{proof}

\section{Proof of Theorem \ref{mainTeo}}
\label{Proof}
In this section, we prove Theorem \ref{mainTeo}.
Note first that for any $E\subset I$,
\begin{equation}
    \nonumber
    \osc(|f|;E)\le\osc(f;E).
\end{equation}
Consequently, $\mathcal{V}_p^\alpha(|f|;I)\le\mathcal{V}_p^\alpha(f;I)$ for any $1\le p<\infty$ and $0\le\alpha\le 1/p'$.
Furthermore, since $f^*=|f|^*$, we may assume without loss of generality that $f\ge0$ throughout the proofs below. For the reader's convenience, we shall restate this assumption whenever appropriate.

We first consider the case of Theorem \ref{mainTeo} with zero smoothness (i.e. $\alpha=0$). The proof is immediate, but we include it for the sake of completeness.
\begin{prop}
    \label{verySimplePS}
    Assume that $1\le p<\infty$ and $f\ge0$. Then
    \begin{equation}
        \label{verySimplePS-ineq}
        \mathcal{V}_p(f^*;I^*)\le\mathcal{V}_p(f;I).
    \end{equation}
\end{prop}
\begin{proof}
    In order for (\ref{verySimplePS-ineq}) to be non-trivial, assume that $f\in\mathcal{BV}_p(I)$. 
    Note first that if $g:[a,b]\rightarrow\mathbb{R}$ is monotone on $[a,b]$, then for any $1\le p<\infty$
    \begin{equation}
        \nonumber
        \mathcal{V}_p(g;[a,b])=|g(a)-g(b)|.
    \end{equation}
    Indeed, this is an immediate consequence of the fact that for $p\ge1$ and any $a_1,a_2,\ldots, a_n\ge0$, there holds
    \begin{equation}
        \nonumber
        \sum_{j=1}^na_j^p\le\left(\sum_{j=1}^na_j\right)^p.
    \end{equation}
    Consequently,
    \begin{equation}
        \nonumber
        \mathcal{V}_p(f^*;I^*)=\esssup_{x\in I}f(x)-\essinf_{x\in I}f(x).
    \end{equation}
    On the other hand, given any $\epsilon>0$, there are points $x_0,y\in I$ such that
    \begin{equation}
        \nonumber
        f(x_0)>\esssup_{x\in I}f(x)-\epsilon,\quad f(y_0)<\essinf_{x\in I}f(x)+\epsilon.
    \end{equation}
    It easily follows that
    \begin{equation}
        \nonumber
        \mathcal{V}_p(f;I)\ge f(x_0)-f(y_0)>\esssup_{x\in I}f(x)-\essinf_{x\in I}f(x)-2\epsilon>\mathcal{V}_p(f^*;I^*)-2\epsilon.
    \end{equation}
    Letting $\epsilon\rightarrow0$ completes the proof of (\ref{verySimplePS-ineq}).
\end{proof}
We proceed with the case of positive smoothness. In order to have $\alpha>0$, we must have $p>1$. Note also that in this case, we may assume that $f\in C(I)$ (since $f\in{\rm Lip}_\alpha(I)$, by Proposition \ref{contProp}). As before, we also assume $f\ge0$. 
We first prove a variant of Theorem \ref{mainTeo} with an additional assumption on $f$.
\begin{prop}
    \label{simplePS}
    Let $1<p<\infty$ and $0<\alpha\le 1/p'$. Assume that $f\in C(I)$, $f\ge0$ on $I$, and 
    \begin{equation}
        \label{simplePSA3}
        |\{x\in I:f(x)=y\}|=0\quad\text{for every } y\ge0.
    \end{equation}
    Then 
    \begin{equation}
        \label{simplePS-ineq}
        \mathcal{V}_p^\alpha(f^*;I^*)\le\mathcal{V}_p^\alpha(f;I).
    \end{equation}
\end{prop}
\begin{proof}
    The assumption (\ref{simplePSA3}) implies that $f^*$ is strictly decreasing, by Lemma \ref{decreasingLem}. Further, $f$ and $f^*$ are equimeasurable. Therefore, for any $t\in I^*$
    \begin{equation}
        \label{simplePS-pf1}
        \mu_f(f^*(t))=\mu_{f^*}(f^*(t))=t.
    \end{equation}
    Take an arbitrary finite family of disjoint intervals $[a_k,b_k]\subset I^*~~(1\le k\le n)$ and consider
    $$
    S:=\sum_{k=1}^n\frac{(\osc(f^*;(a_k,b_k))^p}{(b_k-a_k)^{\alpha p}}=\sum_{k=1}^n\frac{(f^*(a_k)-f^*(b_k))^p}{(b_k-a_k)^{\alpha p}}.
    $$
     We shall find disjoint intervals $I_k~~(1\le k\le n)$ such that
    \begin{equation}
        \label{ineq1}
        \frac{(f^*(a_k)-f^*(b_k))^p}{(b_k-a_k)^{\alpha p}}\le\frac{\osc(f;I_k)^p}{|I_k|^{\alpha p}}\quad(1\le k\le n).
    \end{equation}
    Hence $S\le\mathcal{V}_p^\alpha(f;I)$ and  (\ref{simplePS-ineq}) follows.
    Fix $k\in\{1,2,\ldots,n\}$. Define
    \begin{equation}
        \nonumber
        M_k=\{x\in I: f^*(a_k)>f(x)>f^*(b_k)\}.
    \end{equation}
    Note that the sets $M_k$ are disjoint. Further, since 
    \begin{equation}
        \nonumber
        M_k=\{x\in I:f(x)> f^*(b_k)\}\setminus\{x\in I:f(x)\ge f^*(a_k)\}
    \end{equation}
    we have
    \begin{equation}
        \nonumber
        |M_k|=\mu_f(f^*(b_k))-\left(\mu_f(f^*(a_k))+|\{x\in I:f(x)=f^*(a_k)\}|\right)
    \end{equation}
    By (\ref{simplePSA3}), $|\{x\in I:f(x)=f^*(a_k)\}|=0$. Therefore, by (\ref{simplePS-pf1})
    $$
        |M_k|=\mu_f(f^*(b_k))-\mu_f(f^*(a_k))=b_k-a_k
    $$
    Since $f$ is continuous, we may invoke Lemma \ref{mainLemma} to obtain an open interval $I_k\subset M_k$ such that
    \begin{equation}
        \nonumber
        f^*(a_k)-f^*(b_k)={\rm osc}(f;I_k).
    \end{equation}
    Since $I_k\subset M_k$, there also holds $|I_k|\le b_k-a_k$, so
    \begin{equation}
        \nonumber
        \frac{(f^*(a_k)-f^*(b_k))^p}{(b_k-a_k)^{\alpha p}}\le\frac{\osc(f;I_k)^p}{|I_k|^{\alpha p}}.
    \end{equation}
    This concludes the proof.
\end{proof}

We shall use the following perturbation lemma. It was proved in \cite{BaLi26b}.
\begin{lem}
    \label{LipschitzApprox}
    Let $f\in C(I)$. For any $\epsilon>0$ there exists $g_\epsilon\in{\rm Lip}(I)$ such that $|g_\epsilon|_{{\rm Lip}(I)}<\epsilon$ and for every $y\in\mathbb{R}$ there holds
    \begin{equation}
        \nonumber
        |\{x\in I:f(x)+g_\epsilon(x)=y\}|=0.
    \end{equation}
\end{lem}
Essentially, the previous lemma asserts that any continuous function $f$ can be perturbed by an arbitrarily small function $g_\epsilon\in{\rm Lip}(I)$ so that the result has no level set of positive measure.

\begin{proof}[Proof of Theorem \ref{mainTeo}]
The case $\alpha=0$ and $1\le p<\infty$ was proved in Proposition \ref{verySimplePS}. Assume that $p>1$ and $\alpha>0$. In the case of positive smoothness, we assume without loss of generality that $f\in C(I)$, since $\mathcal{BV}_p^\alpha(I)\hookrightarrow{\rm Lip}_\alpha(I)$ by Proposition \ref{contProp}. As before, we may also assume $f\ge0$.
Let $\epsilon>0$ be arbitrary and take $g_\epsilon$ as given by Lemma \ref{LipschitzApprox}. 
Define
\begin{equation}
    \nonumber
    f_\epsilon(x)=f(x)+g_\epsilon(x)+\epsilon,
\end{equation}
then $f_\epsilon$ is continuous and $f_\epsilon(x)\ge g_\epsilon(x)+\epsilon\ge0$ for all $x\in I$.
By Lemma \ref{LipschitzApprox}, $f_\epsilon$ additionally satisfies (\ref{simplePSA3}). Hence, by Proposition \ref{simplePS}, 
\begin{equation}
    \label{mainTeoEq1}
    \mathcal{V}_p^\alpha(f^*_\epsilon;I^*)\le\mathcal{V}_p^\alpha(f_\epsilon;I).
\end{equation}
Furthermore, since $|g_\epsilon|_{{\rm Lip}(I)}<\epsilon$, we have $\osc(g_\epsilon;J)\le \epsilon|J|$ for any interval $J\subset I$. Observe also that $p-p\alpha\ge1$, whence
\begin{equation}
    \label{mainTeoEq2}
    \mathcal{V}_p^\alpha(g_\epsilon;I)\le\sup\left(\sum_{I_j}\epsilon^p|I_j|^{p-\alpha p}\right)^{1/p}\le\epsilon(b-a)^{1/p}.
\end{equation}
By (\ref{mainTeoEq1}) and (\ref{mainTeoEq2}),
\begin{equation}
    \nonumber
    \mathcal{V}_p^\alpha(f_\epsilon^*;I^*)\le \mathcal{V}_p^\alpha(f_\epsilon;I)\le\mathcal{V}_p^\alpha(f;I)+\mathcal{V}_p^\alpha(g_\epsilon;I)\le\mathcal{V}_p^\alpha(f;I)+\epsilon(b-a)^{1/p}.
\end{equation}
As $\epsilon\rightarrow0$, $f_\epsilon\rightarrow f$ uniformly. Using the well-known inequality $\|f^*_\epsilon-f^*\|_\infty\le\|f_\epsilon-f\|_\infty$, it follows that $f^*_\epsilon\rightarrow f^*$ uniformly. The proof is concluded by invoking (\ref{semicontinuity}):
$$
\mathcal{V}_p^\alpha(f^*;I^*)\le\liminf_{\epsilon\rightarrow0}\mathcal{V}_p^\alpha(f^*_\epsilon;I^*)\le
\liminf_{\epsilon\rightarrow0}\left(\mathcal{V}_p^\alpha(f;I)+\epsilon(b-a)^{1/p}\right)=
\mathcal{V}_p^\alpha(f;I).
$$
This concludes the proof of Theorem \ref{mainTeo}.
\end{proof}

\section{Additional results}
\label{AddComm}

In this section, we discuss some additional results and applications of Theorem \ref{mainTeo} and the methods used to prove it.

\subsection{Nowhere differentiable functions}

The classical P\'{o}lya-Szeg\H{o} principle (\ref{polya-szego-classical}) is useful only if $f\in W^{1,p}(I)~~(p\ge1)$. Theorem 1.2, however, applies to functions possessing a suitable degree of fractional smoothness, namely, functions belonging to $\mathcal{BV}_p^\alpha(I)$ for some $\alpha>0$. In this subsection, we illustrate this using two classical nondifferentiable functions.

Define
\begin{equation}
    \nonumber
    \psi(x)=\min_{z\in\mathbb{Z}}|x-z|
\end{equation}
and take $b\in\mathbb{N}, b\ge2$. The 
\emph{Takagi-van der Waerden function} is defined by
\begin{equation}
        \nonumber
        W(x)=\sum_{k=1}^\infty\frac{\psi(b^kx)}{b^k}.
\end{equation}
It is well-known that $W\in{\rm Lip}_\alpha([0,1])$ for every $\alpha\in(0,1)$ (see e.g. \cite{ShSa90}). At the same time, $W$ is nowhere differentiable. Fix $p\in(1,\infty)$ and let $\alpha\in(0,1/p')$. Since $\alpha+1/p<1$, we have $W\in{\rm Lip}_{\alpha+1/p}([0,1])$. Therefore, by Proposition \ref{contProp} 
\begin{equation}
    \nonumber
    \mathcal{V}_p^\alpha(W;[0,1])<\infty
\end{equation}
for any $1<p<\infty$ and any $\alpha\in[0,1/p')$. By Theorem \ref{mainTeo}, 
\begin{equation}
    \nonumber
    \mathcal{V}_p^\alpha(W^*;[0,1])\le \mathcal{V}_p^\alpha(W;[0,1]),
\end{equation}
which is a fractional variant of the P\'{o}lya-Szeg\H{o} inequality for $W$.

Another interesting case is \emph{Riemann's ''nondifferentiable'' function} $R$ defined as
\begin{equation}
    \label{riemann}
    R(x)=\sum_{n=1}^\infty\frac{\sin(\pi n^2x)}{n^2}.
\end{equation}
According to Weierstrass, Riemann suggested that (\ref{riemann}) is an example of a continuous and nowhere differentiable function. As it turns out, the smoothness properties of $R$ are quite involved, see the works cited in \cite{Li26}. Variational properties of $R$ were investigated in \cite{Li26}, where we proved that
\begin{equation}
    \label{p-var-R}
    R\notin\mathcal{BV}_p([0,2])\text{ for }p<4/3\quad\text{and}\quad R\in\mathcal{BV}_p([0,2])\text{ for }p>4/3.
\end{equation}
Further, it is simple to verify that $R\in{\rm Lip}_{1/2}([0,2])$ (see \cite{Li26}). We now have
\begin{equation}
    \nonumber
    \mathcal{V}_p^\alpha(R^*;[0,2])\le\mathcal{V}_p^\alpha(R;[0,2])
\end{equation}
whenever $(p,\alpha)$ satisfies
\begin{equation}
    \label{parameterChoice}
    p>\frac{4}{3}\quad\text{and}\quad0\le \alpha<\frac{1}{2}-\frac{2}{3p}.
\end{equation}
Indeed, take any $p>4/3$ and $\alpha<1/2-2/(3p)$.
We shall prove that $R\in\mathcal{BV}_p^\alpha([0,2])$. Define $p_0=p-2p\alpha$, then $4/3<p_0<p$. Since $p_0>4/3$, we have $R\in\mathcal{BV}_{p_0}([0,2])$. Furthermore, since $R\in{\rm Lip}_{1/2}([0,2])$, it follows that $\osc(R;J)\le C|J|^{1/2}$ for any $J\subset[0,2]$. Thus,
    \begin{eqnarray}
    \nonumber
    \left(\sum_k\frac{\osc(R;I_k)^p}{|I_k|^{\alpha p}}\right)^{1/p}&\le&\sup_k\frac{\osc(R;I_k)^{1-p_0/p}}{|I_k|^{\alpha}}\left(\sum_k\osc(R;I_k)^{p_0}\right)^{1/p}\\
    \nonumber
    &\le&C\sup_k|I_k|^{1/2-p_0/(2p)-\alpha}\mathcal{V}_{p_0}(R;[0,2])^{p_0/p}\\
    \nonumber
    &=&C\mathcal{V}_{p_0}(R;[0,2])^{p_0/p}<\infty,
    \end{eqnarray}
where we used $p_0/(2p)=1/2-\alpha$.

We do not claim that the parameter range for $\alpha$ in (\ref{parameterChoice}) is sharp. However, as $p\rightarrow4/3+$, the smoothness $\alpha$ must tend to 0, as it does in (\ref{parameterChoice}). In fact, for $p=4/3$ there is an easy argument to show that $R\notin\mathcal{BV}^\beta_{4/3}([0,2])$ for any $\beta>0$. Indeed, it was proved in \cite{KoLi12} that for $1<p<q<\infty$, the embedding
\begin{equation}
    \label{embedding}
    \mathcal{BV}_q^{1/p-1/q}(I)\hookrightarrow\mathcal{BV}_p(I)
\end{equation}
holds. Let $q=4/3$ and fix $\beta\in(0,1/4)$. (The argument can be modified to handle $\beta\ge1/4$ as well.) Choose $p>1$ such 
that
\begin{equation}
    \nonumber
    \beta+\frac{3}{4}=\frac{1}{p}.
\end{equation}
Then $\beta=1/p-1/q$, and hence, by (\ref{embedding})
\begin{equation}
    \nonumber
    R\in\mathcal{BV}_{4/3}^\beta([0,2])\quad\Rightarrow\quad R\in\mathcal{BV}_p([0,2]).
\end{equation}
Since $\beta>0$, we have $1/p>3/4$, or equivalently, $p<4/3$. This contradicts (\ref{p-var-R}), and hence $R\notin\mathcal{BV}_{4/3}^\beta([0,2])$ for any $\beta>0$. The interesting question of whether $R\in\mathcal{BV}_{4/3}([0,2])$ was left open in \cite{Li26}.

\subsection{Moduli of continuity}
For a function $f\in L^p(I)$, the \emph{$L^p$ modulus of continuity} of $f$ is defined by
\begin{equation}
    \nonumber
    \omega(f;\delta)_p=\sup_{0<h\le\delta}\left(\int_a^{b-h}|f(x+h)-f(x)|^p\mathrm{d}x\right)^{1/p}.
\end{equation}
Ul'yanov posed the problem of estimating the modulus of continuity of the nonincreasing rearrangement $f^*$ in terms of that of $f$. In a series of papers, Oswald, Wik, and Brudnyi independently proved that there exists a constant $\gamma_p\ge1$ such that the inequality 
\begin{equation}
    \label{intMod1}
    \omega(f^*;\delta)_p\le \gamma_p\omega(f;\delta)_p
\end{equation}
holds for $p\in[1,\infty)$ and $\delta\in(0,|I|/2]$. We refer the reader to \cite[Section 4.2]{Ko07} for a survey of these and related results. To the best of our knowledge, the optimal value of the constant $\gamma_p$ remains unknown. We shall discuss an analogue of (\ref{intMod1}) for a moduli of continuity defined in terms of $p$-variation. Given any family $\mathcal{I}=\{I_k\}$ of disjoint intervals contained in $I$, denote by 
\begin{equation}
    \nonumber
    \|\mathcal{I}\|=\sup_j|I_j|.
\end{equation}
The \emph{modulus of $p$-continuity} of $f$ is defined by
\begin{equation}
    \label{pmodDef}
    \omega_{1-1/p}(f;\delta)=\sup_{\|\mathcal{I}\|\le\delta}\left(\sum_{I_j\in\mathcal{I}}\osc(f;I_j)^p\right)^{1/p}\quad(0<\delta\le|I|).
\end{equation}
The definition (\ref{pmodDef}) was given by Terekhin \cite{Te67}.
For $p>1$, the quantity $\omega_{1-1/p}(f;\delta)$ may tend to 0 as $\delta\rightarrow0$ for non-trivial continuous functions. 

We shall show that the non-increasing rearrangement does not increase $\omega_{1-1/p}(f;\delta)$. In contrast to (\ref{intMod1}), the corresponding sharp constant is 1.
\begin{prop}
    Assume that $1<p<\infty$ and $f\in C(I)$. Then for any $\delta\in(0,|I|]$, there holds
    \begin{equation}
        \label{pModIneq}
        \omega_{1-1/p}(f^*;\delta)\le\omega_{1-1/p}(f;\delta).
    \end{equation}
\end{prop}
\begin{proof}
    We only sketch the proof; the argument is very similar to the proof of Proposition \ref{simplePS}.
    Assume first that $f$ satisfies (\ref{simplePSA3}). Fix $\delta>0$ and take arbitrary disjoint intervals $(a_k,b_k)\subset I^*~~(1\le k\le n)$ with $b_k-a_k\le\delta$ for $1\le k\le n$. Consider
    \begin{equation}
        \nonumber
        S=\sum_{k=1}^n(f^*(a_k)-f^*(b_k))^p.
    \end{equation}
    In the same fashion as in the proof of Proposition \ref{simplePS}, we obtain disjoint intervals $\{I_k:1\le k\le n\}$ contained in $I$ such that
    \begin{equation}
        \nonumber
        f^*(a_k)-f^*(b_k)=\osc(f;I_k),
    \end{equation}
    and $|I_k|\le b_k-a_k\le\delta$. Consequently,
    \begin{equation}
        \nonumber
        S=\sum_{k=1}^n\osc(f;I_k)^p\le\omega_{1-1/p}(f;\delta)^p.
    \end{equation}
    The assumption (\ref{simplePSA3}) is removed by combining Lemma \ref{LipschitzApprox} with a straightforward modification of (\ref{semicontinuity}).
\end{proof}

\subsection{A nonlinear Fredholm integral equation}

In \cite{EnHoZe93}, the authors study the class of nonlinear Fredholm integral equations of the first kind
\begin{equation}
    \label{fredholm}
    \int_0^1k(s,x(t))\mathrm{d}t=y(s)\quad s\in[0,1].
\end{equation}
A characteristic feature of (\ref{fredholm}) is that, for a given right-hand side $y$, the solutions $x(t)~~(0\le t\le1)$ are in general not unique. To address this issue, the authors observe that whenever $x$ solves (\ref{fredholm}), so does the non-increasing rearrangement $x^*$. Furthermore, conditions are identified in order for injectivity of the operator defined by (\ref{fredholm}) in the class of monotone function. Combining this observation with the fact that $x\in X\Rightarrow x^*\in X$  for both $X=C([0,1])$ and $X=W^{1,p}([0,1])~~(p\ge1)$,
they obtain a canonical representative of each rearrangement class of solutions. In this way, they establish a natural notion of uniqueness modulo rearrangements of solutions to (\ref{fredholm}) in both $C([0,1])$ and $W^{1,p}([0,1])~~(p\ge1)$. By Theorem \ref{mainTeo}, the same strategy applies to $\mathcal{BV}_p^\alpha([0,1])$ for $1\le p<\infty$ and $0\le\alpha\le 1/p'$.  This provides additional spaces in which one could assign canonical solutions to \eqref{fredholm} while retaining uniqueness modulo rearrangements. 

\bibliographystyle{plain}
\bibliography{ref2}
\end{document}